\documentclass[12pt]{article}
\usepackage{amsfonts}
\usepackage{amsmath}
\usepackage{mathrsfs}
\usepackage{color}
\usepackage{tikz}
\usepackage{authblk}
\usepackage{mathrsfs,amscd,amssymb,amsthm,amsmath,bm,graphicx,psfrag,subfigure,url,verbatim}

\allowdisplaybreaks
\usepackage{indentfirst}
\newtheorem{thm}{Theorem}
\newtheorem{prop}{Proposition}

\newtheorem{claim}{Claim}

\newtheorem{lem}{Lemma}
\newtheorem{cor}{Corollary}

\newenvironment{wst}
{\setlength{\leftmargini}{1.5\parindent}
 \begin{itemize}
 \setlength{\itemsep}{-1.1mm}}
{\end{itemize}}

\begin{document}
	
\title{\bf The maximum sum of the sizes of all intersections within $m$-size families}
\author[a]{Sumin Huang\thanks{		Email: sumin2019@sina.com}}
\author[b]{Gyula O.H. Katona\thanks{Corresponding author.
Email: katona.gyula.oh@renyi.hu}}
\author[b,c]{Erfei Yue\thanks{Email: yef9262@mail.bnu.edu.cn}}
\affil[a]{School of Mathematics and Statistics, Nanjing University of Information Science and Technology, Nanjing 210044, P.R. China}
\affil[b]{Alfr\'{e}d R\'{e}nyi Institute of Mathematics, Hungarian Academy of Sciences, Budapest 1053, Hungary}
\affil[c]{Institute of Mathematics, E\"{o}tv\"{o}s Lor\'{a}nd University, Budapest 1117, Hungary}

\date{}
\maketitle

\begin{abstract}
	For a family of sets $\mathcal{F}$, let $\omega(\mathcal{F}):=\sum_{\{A,B\}\subset \mathcal{F}}|A\cap B|$. In this paper, we prove that provided $n$ is sufficiently large, for any $\mathcal{F}\subset \binom{[n]}{k}$ with $|\mathcal{F}|=m$, $\omega(\mathcal{F})$ is maximized by the family consisting of the first $m$ sets in the lexicographical ordering on $\binom{[n]}{k}$. Compared to the maximum number of adjacent pairs in families, determined by Das, Gan and Sudakov in 2016, $\omega(\mathcal{F})$ distinguishes the contributions of intersections of different sizes. Then our results is an extension of Ahlswede and Katona's results in 1978, which determine the maximum number of adjacent edges in graphs. Besides, since $\omega(\mathcal{F})=\frac{1}{2}\left(\sum_{x\in [n]}|\{F\in \mathcal{F}:x\in F\}|^2-km\right)$ for $k$-uniform family of size $m$, our results also give a sharp upper bound of the sum of squares of degrees in a hypergraph.
\end{abstract}

\vspace{2mm} \noindent{\bf Keywords}: intersection, squares of degree, family, lexicographic ordering
\vspace{2mm}

\setcounter{section}{0}
	
\section{Introduction}
    Let $[n]=\{1,\ldots,n\}$ and $\binom{X}{k}$ be the family of all $k$-subsets of a set $X$ for $|X|\geq k\geq 0$. We say a family of sets $\mathcal{F}$ is $k$-uniform if $\mathcal{F}\subset \binom{[n]}{k}$. A family $\mathcal{F}$ is \textit{intersecting} if there exists no disjoint pair $F_1,F_2\in\mathcal{F}$. As one of the most fundamental theorems in extremal set theory, Erd\H{o}s, Ko and Rado \cite{Erdos} determined the largest size of $k$-uniform intersecting families. Obviously, when the size of a family is sufficiently large, there must exist disjoint pairs. It is natural to ask how many disjoint pairs must appear in a family $\mathcal{F}$ of size $m$? 
    
    Actually, this problem is the {\it Erd\H{o}s-Rademacher problem} or {\it supersaturation} with respect to intersecting families, which asks how many copies of the forbidden configuration must appear in a structure larger than the extremal bound. The Erd\H{o}s-Rademacher problem with respect to intersecting families was first investigated by Frankl \cite{Frankl} and Ahlswede \cite{Ahlswede} independently. They considered the minimum number of disjoint pairs in a non-uniform family $\mathcal{F}$ of size $m$. Then, in \cite{Ahlswede}, Ahlswede asked the corresponding problem to $k$-uniform family.

    Before proposing this problem, Ahslwede and Katona \cite{Ahlswede2} had already solved the case $k=2$. They determined the maximum number of adjacent edges in a graph of size $m$. Let $C_{n}^m$ be the graph with vertex set $[n]$ and edge set $\{(i,j):1\leq i<j\leq a\}\cup \{(i,a+1):1\leq i\leq b\}$, where $a$ and $b$ are determined by the unique representation $m=\binom{a}{2}+\binom{b}{1}$ with $0\leq b<a$. Also, let $S^m_n$ be the complement of $C^{\binom{n}{2}-m}_n$.
    
    \begin{thm}[\cite{Ahlswede2}]\label{thm-ahl}
    	Let $n$ and $m$ be positive integers. Then either $C_{n}^m$ or $S_{n}^m$ maximizes the number of adjacent pairs among all $n$-vertex graphs of size $m$.
    \end{thm}
    
    In 2003, Bollob\'{a}s and Leader \cite{Bollobas} gave a new proof to the special case $m=\sum_{i\geq k}\binom{n}{i}$ for Frankl and Ahlswede's result. They also conjectured that for $k$-uniform families with small size, the initial segment of the lexicographical ordering on $\binom{[n]}{k}$ minimize the number of disjoint pairs. This conjecture was confirmed by Das, Gan and Sudakov \cite{Das} in 2016. The {\it lexicographical ordering} of sets is defined by $A$ being smaller than $B$ if and only if the minimum element in $A\backslash B$ is smaller than that in $B\backslash A$. Denote by $\mathcal{L}_{n,k}^m$ the family consisting of the first $m$ sets in the lexicographical ordering on $\binom{[n]}{k}$.
    
    \begin{thm}[\cite{Das}]\label{thm-das}
    	Provided $n>108k^2l(k+l)$ and $0\leq m\leq \binom{n}{k}-\binom{n-l}{k}$, $\mathcal{L}_{n,k}^m$ minimizes the number of disjoint pairs among all families in $\binom{[n]}{k}$ of size $m$.
    \end{thm}

    Note that Theorem~\ref{thm-das} shows that $\mathcal{L}_{n,k}^m$ maximizes the number of adjacent pairs in a $k$-uniform family of size $m$. In this statement, an adjacent pair $(F_1,F_2)$ with $|F_1\cap F_2|=1$ and another adjacent pair $(F'_1,F'_2)$ with $|F'_1\cap F'_2|=k-1$ have the same contribution to enumerate the number of adjacent pairs. Distinguishing their contributions is a natural extension of Theorem~\ref{thm-das}. For a family $\mathcal{F}$, denote the sum of sizes of all intersections by
    $$
    \omega(\mathcal{F}):=\sum_{\{A,B\}\subset \mathcal{F}}|A\cap B|.
    $$ 
    Then $\omega(\mathcal{F})$ is a function which distinguishes the contribution of intersections with different sizes.
    
    $\omega(\mathcal{F})$ was first proposed by Kong and Ge \cite{Kong} in an unpublished paper.
    
    \begin{thm}[\cite{Kong}]
    	Let $C_0\geq 3\times 10^3$ be an absolute constant and $k\geq 2, r\geq 0$ be two fixed integers. For any $n\geq C_0(r+1)^3(k+r)k^2$ and $\delta\in [\frac{150k^3(r+1)^2}{n},1-\frac{150k^3(r+1)^2}{n}]\cup \{1\}$, if $\mathcal{F}\subset \binom{[n]}{k}$ with $|\mathcal{F}|=\sum_{i=1}^{r}\binom{n-i}{k-1}+\delta \binom{n-(r+1)}{k-1}$ is an extremal family maximizing $\omega(\mathcal{F})$, then
    	$$\mathcal{L}_{n,k}^{m_1}\subseteq \mathcal{F}\subseteq \mathcal{L}_{n,k}^{m_2},$$
    	where $m_1:=\sum_{i=1}^{r}\binom{n-i}{k-1}$ and $m_2:=\sum_{i=1}^{r+1}\binom{n-i}{k-1}$.
    \end{thm}
    They only provided a property of extremal families, so in this paper, we attempt to uniquely determine the extremal family.
    
    Note that for $2$-uniform families $\mathcal{F}$, $\mathcal{F}$ minimizes the number of disjoint pairs if and only if $\mathcal{F}$ maximizes $\omega(\mathcal{F})$. Thus, Theorem~\ref{thm-ahl} gives the extremal family of $\omega(\mathcal{F})$ provided $\mathcal{F}\subset \binom{[n]}{2}$. However, this statement is not true for $k\geq 3$. While Theorem~\ref{thm-das} focus on the minimum number of disjoint pairs, we will determine the maximum value of $\omega(\mathcal{F})$ in this paper.

    For $\mathcal{F}\subseteq \binom{[n]}{k}$ and $x\in [n]$, let $\mathcal{F}(x):=\{F\in \mathcal{F}:x\in F\}$. Also, for $\mathcal{F}\subset \binom{[n]}{k}$ of size $m$, it can be confirmed that $\omega(\mathcal{F})=\sum_{x\in [n]} \binom{|\mathcal{F}(x)|}{2}=\frac{1}{2}\sum_{x\in [n]} |\mathcal{F}(x)|^2-\frac{1}{2}km$. Thus, in order to maximize $\omega(\mathcal{F})$, it is sufficient to maximize $\sum_{x\in [n]} |\mathcal{F}(x)|^2$, that is the sum of squares of degrees in a hypergraph. For a graph $G$ of size $m$, de Caen \cite{Caen} gave an upper bound $\frac{2m^2}{n-1}+(n-1)m$ of the sum of squares of degrees and then this bound was improved by Das \cite{chDas}. In 2003, Bey \cite{Bey} generalized de Caen's upper bound to hypergraph. We state Bey's  results using the notation of families of sets and the fact $\omega(\mathcal{F})=\frac{1}{2}\sum_{x\in [n]} |\mathcal{F}(x)|^2-\frac{1}{2}km$.
    
    \begin{thm}[\cite{Bey}]\label{thm-bey}
    	Let $n,k,m$ be positive integers. If $0<m\leq \binom{n}{k}$, then for any $\mathcal{F}\subset \binom{[n]}{k}$ with $|\mathcal{F}|=m$,
    	$$
    	\omega(\mathcal{F})\leq \frac{k(k-1)}{2(n-1)}m^2+\frac{1}{2}\binom{n-2}{k-1}m-\frac{1}{2}km,
    	$$
        where the equality holds if and only if $\mathcal{F}$ is one of $\binom{[n]}{k}$ , $\{F\in \binom{[n]}{k}:1\in F\}$, $\{F\in \binom{[k+1]}{k}:[r]\subset F\}$ for each $r=2,\dots, \left\lfloor\frac{k+1}{2}\right\rfloor$, or one of the complement of these families.
    \end{thm}

    Note that Theorem~\ref{thm-bey} gives a sharp upper bound if and only if one of the following statements holds:
    \begin{wst}
    	\item[\rm (i)] $m\in \{\binom{n}{k},\binom{n-1}{k-1},\binom{n-1}{k}\}$;
    	\item[\rm (ii)] $n=k+1$ and $m\in \{2,3,\ldots,k-2,k-1\}$. 
    \end{wst}
    
    In this paper, we determine an extremal family of $\omega(\mathcal{F})$ for some special $m$. 
    To state our result, we should introduce the {\it $k$-cascade} form of $m$:
    $$
    m=\binom{a_1}{k}+\binom{a_{2}}{k-1}+\cdots+\binom{a_s}{k-s+1},\quad a_1>a_{2}>\cdots>a_s\geq 1.
    $$
    Note that every positive integer has a unique $k$-cascade representation. Now for given $n$ and $k$, we write $\binom{n}{k}-m$ in its $k$-cascade form and replace $a_i$ by $n-r_{i}$:
    \begin{align}
    	\binom{n}{k}-m=\binom{n-r_1}{k}+\binom{n-r_2}{k-1}+\cdots+\binom{n-r_{s}}{k-s+1},1\leq r_1<\cdots<r_{s}\leq n-1.\label{eqc}
    \end{align}
    We call $(r_1,\ldots,r_s)$ the {\it $(n,k)$-cascade coefficients} of $m$. 
    For $2\leq t\leq s$, let 
    \begin{align}
    	\delta_{n,k,m,t}:=\frac{\sum_{i=t}^s\binom{n-r_i}{k-i+1}}{\binom{n-t+1}{k-t+1}}. \label{eqd}
    \end{align}
     
    In Section 4, we will give a combinatorial explanation of $(n,k)$-cascade coefficients and $\delta_{n,k,m,t}$. Based on these definitions, we determine the extremal family maximizing $\omega(\mathcal{F})$ in what follows, and we will prove Theorem~\ref{thm-main} in Section 3.
    \begin{thm}\label{thm-main}
        Let $n,k,r$ be positive integers with $n>36k(2k+1)(k+r)r^2$. Assume that the $(n,k)$-cascade coefficients of $m$ are $(r_1,\ldots,r_s)$. For any $\mathcal{F}\subset \binom{[n]}{k}$ with $|\mathcal{F}|=m$, if $r_s\leq r$ and for each $2\leq t\leq s$, $$\frac{12k(2k+1)(k+1)}{n-r}<\delta_{n,k,m,t}<1-\frac{12k(2k+1)(k+1)}{n-r},$$ then $\omega(\mathcal{F})\leq \omega(\mathcal{L}_{n,k}^{m})$, with equality if and only if $\mathcal{F}= \mathcal{L}_{n,k}^{m}$ up to isomorphic.
    \end{thm} 

    The following theorem holds immediately by Theorem~\ref{thm-main} and Lemma~\ref{lem2}, which will be introduced at the end of Section 2. The proof of Theorem~\ref{thm-cor} is omitted in this paper.
    \begin{thm}\label{thm-cor}
        Let $n,k,r$ be positive integers with $n>36k(2k+1)(k+r)r^2$. Assume that the $(n,k)$-cascade coefficients of $m$ are $(r_1,\ldots,r_s)$. For any $\mathcal{F}\subset \binom{[n]}{k}$ with $|\mathcal{F}|=\binom{n}{k}-m$, if $r_s\leq r$ and for each $2\leq t\leq s$,
        $$\frac{12k(2k+1)(k+1)}{n-r}<\delta_{n,k,m,t}<1-\frac{12k(2k+1)(k+1)}{n-r},$$ then $\omega(\mathcal{F})\leq \omega(\binom{[n]}{k}\backslash\mathcal{L}_{n,k}^{m})$, with equality if and only if $\mathcal{F}= \binom{[n]}{k}\backslash\mathcal{L}_{n,k}^{m}$ up to isomorphic.
    \end{thm}
        
    Theorem~\ref{thm-main} and \ref{thm-cor} give a natural extension of Theorem~\ref{thm-ahl}. Note that if the $(n,k)$-cascade coefficients of $m$ have only one element $(r_1)$, then $m=\binom{n}{k}-\binom{n-r_1}{k}$. In this case, we don't need to consider the condition about $\delta_{n,k,m,t}$ in Theorem~\ref{thm-main} and \ref{thm-cor} since $s=1$. Hence, we have the following corollary.
    
    \begin{cor}
    	Let $n,k,r$ be positive integers with $n>36k(2k+1)(k+r)r^2$. For any $\mathcal{F}\subset \binom{[n]}{k}$ of size $m$,
    	\begin{itemize}
    		\item[\rm (i)] if $m\in\{\binom{n}{k}-\binom{n-r_1}{k}:1\leq r_1\leq r\}$, then $\omega(\mathcal{F})\leq \omega(\mathcal{L}_{n,k}^{m})$ with equality if and only if $\mathcal{F}= \mathcal{L}_{n,k}^{m}$ up to isomorphic.
    		\item[\rm (ii)] if $m\in\{\binom{n-r_1}{k}:1\leq r_1\leq r\}$, then $\omega(\mathcal{F})\leq \omega(\binom{[n]}{k}\backslash\mathcal{L}_{n,k}^{m})$ with equality if and only if $\mathcal{F}= \binom{[n]}{k}\backslash\mathcal{L}_{n,k}^{m}$ up to isomorphic.
    	\end{itemize}
    \end{cor}
    
\section{Preliminaries}
    In this section, we introduce some necessary notations and lemmas used in our proof.
    
    For the families $\mathcal{F}$ and $\mathcal{G}$, let $\omega(\mathcal{F},\mathcal{G}):=\sum_{(A,B) \in (\mathcal{F},\mathcal{G})}|A\cap B|$. For a set $F$, we use $\omega(F,\mathcal{F})$ to represent $\omega(\{F\},\mathcal{F})$.
    Given $\mathcal{F}\subseteq \binom{[n]}{k}$ and $x\in [n]$, recall that $\mathcal{F}(x)=\{F\in \mathcal{F}:x\in F\}\subseteq \mathcal{F}$. We call $\mathcal{F}(x)$ a {\it full star} if $\mathcal{F}(x)=\{F\in \binom{[n]}{k}:x\in F\}$. For a subset $X\subseteq [n]$, we say $X$ is a {\it cover} of $\mathcal{F}$ if $F\cap X\neq \emptyset$ for any $F\in \mathcal{F}$. Also, we use the notation $\mathcal{F}_{x}:=\{F\in \mathcal{F}:x \text{ is the minimum element in } F\}$.  The following proposition can be directly derived from the definition.
    \begin{prop}\label{prop1}
    	Let $\mathcal{F}\subseteq \binom{[n]}{k}$ of size $m$. Then we have
    	\begin{wst}
    		\item[\rm (i)] $\sum_{x\in [n]}|\mathcal{F}(x)|=km$.
    		\item[\rm (ii)] $\omega(\mathcal{F})=\sum_{x\in [n]}\binom{|\mathcal{F}(x)|}{2}=\frac{1}{2}\left(\sum_{x\in [n]}|\mathcal{F}(x)|^2-km\right)$.
    		\item[\rm (iii)] If $X$ is a cover of $\mathcal{F}$, then $\cup_{x\in X} \mathcal{F}(x)=\mathcal{F}$. Moreover, if $[r]$ is a cover of $\mathcal{F}$, then $\cup_{x\in [r]} \mathcal{F}_{x}=\mathcal{F}$ and $\sum_{x\in [r]}|\mathcal{F}_x|=|\mathcal{F}|=m$.
    		\item[\rm (iv)] If $F$ is an arbitrary set in $\binom{[n]}{k}$, then $\omega(F,\mathcal{F})=\sum_{x\in F}|\mathcal{F}(x)|$.
    		\item[\rm (v)] $\omega(\mathcal{F},\mathcal{F})=2\omega(\mathcal{F})+\sum_{F\in\mathcal{F}}|F\cap F|=2\omega(\mathcal{F})+km$.
    	\end{wst}
    \end{prop}
    
    Then we introduce a lemma for simplifying calculations.
    
    \begin{lem}\label{lem-simp}
    	Let $\mathcal{F}\subseteq\binom{[n]}{k}$ be a family and $F\in \mathcal{F}$ be a $k$-set. For an arbitrary $k$-set $F'\in \binom{[n]}{k}\backslash\mathcal{F}$, we have
    	$$\omega((\mathcal{F}\backslash \{ F\})\cup \{F'\})-\omega(\mathcal{F})=\omega(F',\mathcal{F}\backslash\{ F\})-\omega(F,\mathcal{F}\backslash\{F\}).$$
    \end{lem}
    \begin{proof}
    	Since $(\mathcal{F}\backslash \{ F\})\cup \{F'\}$ and $\mathcal{F}$ have the same size. By Proposition~\ref{prop1}(v), it is clear that
    	\begin{align}
    		\omega((\mathcal{F}\backslash \{ F\})\cup \{F'\},(\mathcal{F}\backslash \{ F\})\cup \{F'\})-\omega(\mathcal{F},\mathcal{F})=2\omega((\mathcal{F}\backslash \{ F\})\cup \{F'\})-2\omega(\mathcal{F}). \label{eq-0}
    	\end{align}
        Also, by a direct calculation, we have
        \begin{align}
        	&\omega((\mathcal{F}\backslash \{ F\})\cup \{F'\},(\mathcal{F}\backslash \{ F\})\cup \{F'\})\notag\\
        	=& \omega(\mathcal{F}\backslash \{ F\},\mathcal{F}\backslash \{ F\})+2\omega(F',\mathcal{F}\backslash \{ F\})+\omega(F',F')\notag\\
        	=& \omega(\mathcal{F}\backslash \{ F\},\mathcal{F}\backslash \{ F\})+2\omega(F',\mathcal{F}\backslash \{ F\})+|F'\cap F'|\notag\\
        	=& \omega(\mathcal{F}\backslash \{ F\},\mathcal{F}\backslash \{ F\})+2\omega(F',\mathcal{F}\backslash \{ F\})+k.\label{eq-5}
        \end{align}
        Similarly, we can rewrite $\mathcal{F}$ as $(\mathcal{F}\backslash\{F\})\cup \{ F\}$ and then
        \begin{align}
        	\omega(\mathcal{F},\mathcal{F})=\omega(\mathcal{F}\backslash \{ F\},\mathcal{F}\backslash \{ F\})+2\omega(F,\mathcal{F}\backslash \{ F\})+k.\label{eq-6}
        \end{align}
        By \eqref{eq-5} and \eqref{eq-6}, we have 
        $$\omega((\mathcal{F}\backslash \{ F\})\cup \{F'\},(\mathcal{F}\backslash \{ F\})\cup \{F'\})-\omega(\mathcal{F},\mathcal{F})=2\omega(F',\mathcal{F}\backslash \{ F\})-2\omega(F,\mathcal{F}\backslash \{ F\}).$$
        Together with \eqref{eq-0}, we obtain that
        $$\omega((\mathcal{F}\backslash \{ F\})\cup \{F'\})-\omega(\mathcal{F})=\omega(F',\mathcal{F}\backslash\{ F\})-\omega(F,\mathcal{F}\backslash\{F\}),$$
        as desired.
    \end{proof}

    Next, we introduce our main lemma, which shows that in a certain case, 
    $\mathcal{F}$ has a cover of small size.

    \begin{lem}\label{lem1}
    	Let $\mathcal{F}$ be an extremal family in $\binom{[n]}{k}$ of size $m$. Let $r$ be the positive integer such that $\binom{n}{k}-\binom{n-r+1}{k}<m\leq \binom{n}{k}-\binom{n-r}{k}$. Let $\delta:=m/\binom{n-1}{k-1}$ if $r=1$; and $\delta:                                                                                                                                                                                                                                                                                                                                                                                                              =r/3$ if $r\geq 2$.
    	If $\mathcal{F}$ contains no full star, then $\mathcal{F}$ has a cover of $r$ elements provided $\delta n>12k(2k+1)(k+r)r^2$.
    \end{lem}
    \begin{proof}
    	Without loss of generality, we assume that $|\mathcal{F}(1)|\geq |\mathcal{F}(2)|\geq \cdots\geq |\mathcal{F}(n)|$. 
    	Let $$X_0:=\{x\in [n]:|\mathcal{F}(x)|\geq \frac{m}{3kr}\}.$$ 
    	Our main idea is similar to Das, Gan and Sudakov's \cite{Das}, which is to first prove that there are sufficiently many elements in $\mathcal{F}(1)$, and then to show $X_0$ is a cover of $\mathcal{F}$, and finally to confirm that $|X_0|\geq r$.
    	\begin{claim}\label{claim1}
    		$\omega(\mathcal{F})\geq \frac{m^2}{2r}-\frac{1}{2}km$.
    	\end{claim}
        \begin{proof}[Proof of Claim~\ref{claim1}]
        	Since $\mathcal{F}$ is an extremal family, we have $\omega(\mathcal{F})\geq \omega(\mathcal{L}_{n,k}^{m})$. Let $\mathcal{L}:=\mathcal{L}_{n,k}^{m}$. By Proposition~\ref{prop1}(ii), 
        	\begin{align}
        		\omega(\mathcal{L})=&
        		\frac{1}{2}\left(\sum_{x\in [n]}|\mathcal{L}(x)|^2-km\right)
        		\geq \frac{1}{2}\left(\sum_{x\in [r]}|\mathcal{L}(x)|^2-km\right)\notag\\
        		\geq& \frac{1}{2}\left(\sum_{x\in [r]}|\mathcal{L}_x|^2-km\right)
        		\geq \frac{r}{2}\left(\frac{\sum_{x\in [r]}|\mathcal{L}_x|}{r}\right)^2-\frac{1}{2}km.\label{eq-1}
            \end{align}
            Recall that $\binom{n}{k}-\binom{n-r+1}{k}<m\leq \binom{n}{k}-\binom{n-r}{k}$. Then $\mathcal{L}_{n,k}^{m}$ consists of all sets covered by $[r-1]$ and the first $m-\binom{n}{k}+\binom{n-r+1}{k}$ sets with minimum element $r$ in the lexicographic ordering on $\binom{[n]\backslash[r-1]}{k}$. This implies that $[r]$ is a cover of $\mathcal{L}_{n,k}^{m}$. By Proposition~\ref{prop1}(iii), we have $\sum_{x\in [r]}|\mathcal{L}_x|=|\mathcal{L}|=|\mathcal{L}_{n,k}^m|=m$. Then by \eqref{eq-1}, we have
            $$
            \omega(\mathcal{L})\geq \frac{r}{2}\left(\frac{m}{r}\right)^2-\frac{1}{2}km=\frac{m^2}{2r}-\frac{1}{2}km.
            $$
            Thus, $\omega(\mathcal{F})\geq \frac{m^2}{2r}-\frac{1}{2}km$, as desired.
        \end{proof}
    
        \begin{claim}\label{claimm}
        	$m>\frac{\delta n}{k}\binom{n-2}{k-2}$.
        \end{claim}
        \begin{proof}[Proof of Claim~\ref{claimm}]
        First we consider the case $r=1$. It is clear that 
        $$m=\delta\binom{n-1}{k-1}=\frac{\delta(n-1)}{k-1}\binom{n-2}{k-2}> \frac{\delta n}{k}\binom{n-2}{k-2}.$$	
        Then we consider the case $r\geq 2$.	
        Since $\mathcal{L}_{n,k}^m$ contains all sets covered by $[r-1]$, by the inclusion-exclusion principle and the Bonferroni inequalities, we have
        \begin{align}
        	m =|\mathcal{L}^m_{n,k}| \geq & \left|\left\{F\in\binom{[n]}{k}:F\cap [r-1]\neq \emptyset\right\}\right|\notag\\
        	= & \sum_{i=1}^{r-1}(-1)^{i+1}\binom{r-1}{i}\binom{n-i}{k-i}\notag\\
        	\geq & (r-1)\binom{n-1}{k-1}-\binom{r-1}{2}\binom{n-2}{k-2}\notag\\
        	= & \left(\frac{(r-1)(n-1)}{k-1}-\binom{r-1}{2}\right)\binom{n-2}{k-2}.\label{eq-2}
        \end{align}
        It is clear that $\frac{(r-1)(n-1)}{k-1}> \frac{(r-1)(n-1)}{k}$ and $\binom{r-1}{2}<\binom{r}{2}$. Then by \eqref{eq-2}, we have
        $$ m>\left(\frac{(r-1)(n-1)}{k}-\frac{r(r-1)}{2}\right)\binom{n-2}{k-2}=\frac{(r-1)(2n-2-kr)}{2k}\binom{n-2}{k-2}.$$
        To prove Claim~\ref{claimm}, it is sufficient to show that $\frac{(r-1)(2n-2-kr)}{2k}>\frac{\delta n}{k}$, that is $(r-1-\delta)n>\frac{1}{2}(r-1)(kr+2)$. Note that $\delta n>12k(2k+1)(k+r)r^2$ and $\delta=\frac{r}{3}$ when $r\geq 2$, Claim~\ref{claimm} holds immediately.
        \end{proof}
    
        \begin{claim}\label{claim2}
        	$|X_0|<6kr$.
        \end{claim}
        \begin{proof}[Proof of Claim~\ref{claim2}]
        	Suppose $|X_0|\geq 6kr$ and $X$ be a subset of $X_0$ with $|X|=6kr$. Then $\binom{|X|}{2}<18k^2r^2$. By the definition of $X_0$, for each $x\in X\subseteq X_0$, $|\mathcal{F}(x)|\geq \frac{m}{3kr}$. Then we have 
        	\begin{align*}
        		m=|\mathcal{F}|\geq &|\cup_{x\in X} \mathcal{F}(x)|\\
        		\geq &\sum_{x\in X}|\mathcal{F}(x)|-\sum_{\{x,y\}\subset X}|\mathcal{F}(x)\cap \mathcal{F}(y)|\\
        		\geq & |X|\frac{m}{3kr}-\binom{|X|}{2}\binom{n-2}{k-2}\\
        		> & 2m-18k^2r^2\frac{k}{\delta n}m\\
        		= & \left(2-\frac{18k^3r^2}{\delta n}\right)m,
        	\end{align*}
            where the last inequality follows from Claim~\ref{claimm}. Since $\delta n>12k(2k+1)(k+r)r^2>18k^3r^2$, $\left(2-\frac{18k^3r^2}{\delta n}\right)m>m$, which leads to a contradiction. Thus, $|X_0|<6kr$.  
        \end{proof}
        \begin{claim}\label{claim3}
        	$|\mathcal{F}_1|=|\mathcal{F}(1)|> \frac{m}{3r}$.
        \end{claim}
        \begin{proof}[Proof of Claim~\ref{claim3}]
        	Note that if a set contains $1$, then $1$ must be the minimum element in this set. Thus $|\mathcal{F}_1|=|\mathcal{F}(1)|.$
        	
        	By Claim~\ref{claim1} and Proposition~\ref{prop1}(ii), $\sum_{x\in [n]}|\mathcal{F}(x)|^2\geq \frac{m^2}{r}$. Note that for each $x\in [n]$, $|\mathcal{F}(x)|\leq |\mathcal{F}(1)|$. In particular, by the definition of $X_0$, for each $x\notin X_0$, $|\mathcal{F}(x)|< \frac{m}{3kr}$. This implies
        	\begin{align}
        	\frac{m^2}{r}\leq \sum_{x\in X_0}|\mathcal{F}(x)|^2+\sum_{x\notin X_0}|\mathcal{F}(x)|^2< |\mathcal{F}(1)|\sum_{x\in X_0}|\mathcal{F}(x)|+\frac{m}{3kr}\sum_{x\notin X_0}|\mathcal{F}(x)|.\label{eq2}
        	\end{align}
            By Claim~\ref{claim2}, $|X_0|<6kr$ and then $\binom{|X_0|}{2}<18k^2r^2$. Thus, we can obtain that 
            \begin{align}
            \sum_{x\in X_0}|\mathcal{F}(x)|\leq & \left|\bigcup_{x\in X_0}\mathcal{F}(x)\right|+\sum_{\{x,y\}\subset X_0}|\mathcal{F}(x)\cap\mathcal{F}(y)|\notag\\
            \leq & m+\binom{|X_0|}{2}\binom{n-2}{k-2}\notag\\
            < & m+ 18k^2r^2\binom{n-2}{k-2}\notag\\
            \leq & \left(1+\frac{18k^3r^2}{\delta n}\right)m\notag\\
            < & 2m,\label{eq-3}
            \end{align}
            where the penultimate inequality follows from Claim~\ref{claimm} and the last inequality follows from $\delta n>12k(2k+1)(k+r)r^2>18k^3r^2$. 
            Also, by Proposition~\ref{prop1}(i), 
            \begin{align}
        	    \sum_{x\notin X_0}|\mathcal{F}(x)|\leq \sum_{x\in [n]}|\mathcal{F}(x)|=km.\label{eq-4}
            \end{align}
            By \eqref{eq2}, \eqref{eq-3} and \eqref{eq-4}, we have 
            $2m|\mathcal{F}(1)|+\frac{m^2}{3r}>\frac{m^2}{r}$, that is $|\mathcal{F}(1)|>\frac{m}{3r}$.
        \end{proof}
        \begin{claim}\label{claim4}
        	$X_0$ is a cover of $\mathcal{F}$.
        \end{claim}
        \begin{proof}[Proof of Claim~\ref{claim4}]
        	Suppose that there exists a set $F\in \mathcal{F}$ such that $F\cap X_0=\emptyset$. Then $|\mathcal{F}(x)|<\frac{m}{3kr}$ holds for all $x\in F$. 
        	Since $\mathcal{F}$ contains no full star, there exists a set $F_1\in \binom{[n]}{k}$ such that $1\in F_1$ but $F_1\notin \mathcal{F}(1)$. 
        	       	
        	In what follows, we will verify that $\omega((\mathcal{F}\backslash\{F\})\cup \{F_1\})>\omega(\mathcal{F})$, which contradicts the maximality of $\omega(\mathcal{F})$. By Lemma~\ref{lem-simp}, it is sufficient to show that $\omega(F_1,\mathcal{F}\backslash\{F\})>\omega(F,\mathcal{F}\backslash\{F\})$.

        	By Proposition~\ref{prop1}(iv), 
        	$$\omega(F,\mathcal{F})=\sum_{x\in F}|\mathcal{F}(x)|<k\frac{m}{3kr}=\frac{m}{3r}.$$ 
        	Thus
        	\begin{align}
        		\omega(F,\mathcal{F}\backslash\{F\})=\omega(F,\mathcal{F})-|F\cap F|<\frac{m}{3r}-k.\label{eq3}
        	\end{align}
     
        	Again, by Proposition~\ref{prop1}(iv), we have
        	$$
        		\omega(F_1,\mathcal{F}\backslash\{F\})=\sum_{x\in F_1}|(\mathcal{F}\backslash\{F\})(x)|\geq |(\mathcal{F}\backslash\{F\})(1)|\geq |\mathcal{F}(1)|-1.
        	$$
            Note that $|\mathcal{F}(1)|>\frac{m}{3r}$ holds by Claim~\ref{claim3}. We can obtain that
            \begin{align}
            	\omega(F_1,\mathcal{F}\backslash\{F\})>\frac{m}{3r}-1.\label{eq4}
            \end{align}
            By \eqref{eq3} and \eqref{eq4}, it can be calculated that
            \begin{align*}
            	\omega(F_1,\mathcal{F}\backslash\{F\})-\omega(F,\mathcal{F}\backslash\{F\})
            	>2\left(\frac{m}{3r}-1\right)-2\left(\frac{m}{3r}-k\right)
            	\geq 0.
            \end{align*}
            Then by Lemma~\ref{lem-simp}, $\omega((\mathcal{F}\backslash\{F\})\cup \{F_1\})>\omega(\mathcal{F})$. However, according to our assumption, $\mathcal{F}$ is an extremal family, which implies $\omega(\mathcal{F})\geq \omega((\mathcal{F}\backslash\{F\})\cup \{F_1\})$. This leads to a contradiction. Hence $X_0$ is a cover of $\mathcal{F}$, as desired.
        \end{proof}
        Recall that $|\mathcal{F}(1)|\geq |\mathcal{F}(2)|\geq \cdots\geq |\mathcal{F}(n)|$ and $X_0=\{x\in [n]:|\mathcal{F}(x)|\geq \frac{m}{3kr}\}$. So for any $i>j$, $i\in X_0$ implies that $j\in X_0$. Assume that $|X_0|=t$. Then we have $X_0=\{1,2,\ldots,t\}=[t]$ and by Claim~\ref{claim2}, $t<6kr$. Besides, it should be noted that $|\mathcal{F}_1|=|\mathcal{F}(1)|\geq |\mathcal{F}_i|$ for any $i\in [n]$.
        \begin{claim}\label{claim5}
        	For any $i,j\in [t]$, $|\mathcal{F}_i|>|\mathcal{F}_j|-\frac{tk^2}{\delta n}m$.
        \end{claim}
        \begin{proof}[Proof of Claim~\ref{claim5}]
        	For $t=1$, Claim~\ref{claim5} holds immediately. Now we consider the case of $t\geq 2$. First we show that there exists an $F_i\in \mathcal{F}(i)$ such that $F_i\cap [t]=\{i\}$. Otherwise, suppose each set in $\mathcal{F}(i)$ contains at least one element in $[t]$ different from $i$. 
        	
        	By the definition of $\mathcal{F}(i)$, each set in $\mathcal{F}(i)$ must contain $i$. Since this set also contains another element in $[t]$ different from $i$ and such an element can be selected in $(t-1)$ ways, we have $|\mathcal{F}(i)|\leq (t-1)\binom{n-2}{k-2}$.

        	By Claim~\ref{claimm}, $\binom{n-2}{k-2}<\frac{k}{\delta n}m$. Also, using Claim~\ref{claim2}, we have $t<6kr$. Then
        	$$|\mathcal{F}(i)|\leq (t-1)\binom{n-2}{k-2}<\frac{tk}{\delta n}m\leq \frac{6k^2r}{\delta n}m<\frac{m}{3kr},$$
        	where the last inequality holds because $\delta n>12k(2k+1)(k+r)r^2>18k^3r^2$. However, $i\in X_0$ and then $|\mathcal{F}(i)|\geq \frac{m}{3kr}$, which leads to a contradiction. 
        	
        	Let $F_i$ be a set in $\mathcal{F}(i)$ such that $F_i\cap [t]=\{i\}$. Then we have
        	\begin{align}
        		\omega(F_i,\mathcal{F})=\sum_{\substack{j\in [t]\\  
        				j\neq i}}\omega(F_i,\mathcal{F}_j)+\omega(F_i,\mathcal{F}_i)
        			=\sum_{\substack{j\in [t]\\  j\neq i}} \left(\sum_{x\in F_i}|\mathcal{F}_j(x)|\right)+\sum_{x\in F_i}|\mathcal{F}_i(x)|.\label{eq-7}		
        	\end{align}
        	According to the definition of $F_i$, for any $x\in F_i$, $x\neq j$. So for each $F\in \mathcal{F}_j(x)$, $\{x,j\}\subset F$ and then $|\mathcal{F}_j(x)|\leq \binom{n-2}{k-2}$. Therefore, by \eqref{eq-7}, we have
        	\begin{align}
        		\omega(F_i,\mathcal{F})\leq & \sum_{\substack{j\in [t]\\  j\neq i}} \left(\sum_{x\in F_i} \binom{n-2}{k-2}\right)+\sum_{\substack{x\in F_i\\  x\neq i}}|\mathcal{F}_i(x)|+|\mathcal{F}_i(i)|\notag\\
        		\leq & (t-1)k\binom{n-2}{k-2}+ (k-1)\binom{n-2}{k-2}+|\mathcal{F}_i|\notag\\
        		= & |\mathcal{F}_i|+(tk-1)\binom{n-2}{k-2}.\label{eq-8}
        	\end{align}
            Since $\mathcal{F}$ contains no full star, there exists a set $F_j\in \binom{[n]}{k}$ such that $j\in F_j$ but $F_j\notin \mathcal{F}(j)$. It is clear that 
            \begin{align}
            \omega(F_j,\mathcal{F})\geq \omega(F_j,\mathcal{F}_j)=\sum_{x\in F_j}|\mathcal{F}_j(x)|\geq |\mathcal{F}_j(j)|=|\mathcal{F}_j|.\label{eq-9}
            \end{align}
        
            Now by Lemma~\ref{lem-simp}, we obtain that
            \begin{align*}
            	\omega((\mathcal{F}\backslash\{F_i\})\cup\{F_j\})-\omega(\mathcal{F})=&\omega(F_j,\mathcal{F}\backslash\{F_i\})-\omega(F_i,\mathcal{F}\backslash\{F_i\})\\
            	=&\left(\omega(F_j,\mathcal{F})-|F_i\cap F_j|\right)-\left(\omega(F_i,\mathcal{F})-|F_i\cap F_i|\right)\\
            	\geq& |\mathcal{F}_j|-|F_i\cap F_j|-|\mathcal{F}_i|-(tk-1)\binom{n-2}{k-2}+k\\
            	>& |\mathcal{F}_j|-|\mathcal{F}_i|-tk\binom{n-2}{k-2},
            \end{align*}
            where the first inequality holds because of \eqref{eq-8} and \eqref{eq-9}.
            
            Since $\mathcal{F}$ is an extremal family, $\omega(\mathcal{F}\backslash\{F_i\}\cup\{F_j\})\leq \omega(\mathcal{F})$, which implies $|\mathcal{F}_i|>|\mathcal{F}_j|-tk\binom{n-2}{k-2}$.
            Together with Claim~\ref{claimm}, we have $|\mathcal{F}_i|>|\mathcal{F}_j|-\frac{tk^2}{\delta n}m$,
            as desired.
        \end{proof}
    By Claim~\ref{claim4}, $[t]=X_0$ is a cover. Then using Proposition~\ref{prop1}(iii), we have 
    \begin{align}
    	\sum_{i\in [t]}|\mathcal{F}_i|=|\mathcal{F}|=m.\label{eq-12}
    \end{align}
    So there exists an $\mathcal{F}_i$ such that $|\mathcal{F}_i|\leq \frac{m}{t}$. By Claim~\ref{claim5}, 
    \begin{align}
	    |\mathcal{F}_1|<|\mathcal{F}_i|+\frac{tk^2}{\delta n}m<\frac{m}{t}+\frac{tk^2}{\delta n}m.\label{eq5}
    \end{align}
        \begin{claim}\label{claim6}
        	$t< 12r$.
        \end{claim}
        \begin{proof}[Proof of Claim~\ref{claim6}]
        	By Claim~\ref{claim3} and \ref{claim5}, for each $i\in [t]$, 
        	\begin{align*}
        		|\mathcal{F}_i|>&|\mathcal{F}_1|-\frac{tk^2}{\delta n}m
        		=|\mathcal{F}(1)|-\frac{tk^2}{\delta n}m\\
        		>&\frac{m}{3r}-\frac{tk^2}{\delta n}m
        		>\frac{m}{3r}-\frac{6k^3r}{\delta n}m\\
        		>&\frac{m}{3r}-\frac{m}{4r}
        		=\frac{m}{12r},
        	\end{align*}
        	where the penultimate inequality holds because $t<6kr$ and the last inequality holds because $\delta n>12k(2k+1)(k+r)r^2>24k^3r^2$. Then $m=|\mathcal{F}|=\sum_{i\in [t]}|\mathcal{F}_i|>t\frac{m}{12r}$. Thus, $t<12r$.
        \end{proof}
        \begin{claim}\label{claim7}
        	$t\leq r$.
        \end{claim}
        \begin{proof}
        	It can be directly calculated that
        	\begin{align}
        		\omega(\mathcal{F},\mathcal{F})=&\sum_{i\in [t]}\left(\sum_{\substack{j\in [t]\\ j\neq i}}\omega(\mathcal{F}_i,\mathcal{F}_j)\right)+\sum_{i\in [t]}\omega(\mathcal{F}_i,\mathcal{F}_i)\notag\\
        		=&\sum_{i\in [t]}\left(\sum_{\substack{j\in [t]\\ j\neq i}}\sum_{F\in \mathcal{F}_i}\omega(F,\mathcal{F}_j)\right)+\sum_{i\in [t]}\sum_{F\in \mathcal{F}_i}\omega(F,\mathcal{F}_i)\notag\\
        		=&\sum_{i\in [t]}\left(\sum_{\substack{j\in [t]\\ j\neq i}}\sum_{F\in \mathcal{F}_i}\sum_{x\in F}|\mathcal{F}_j(x)|\right)+\sum_{i\in [t]}\sum_{F\in \mathcal{F}_i}\sum_{x\in F}|\mathcal{F}_i(x)|\notag\\
        		=&\sum_{i\in [t]}\sum_{\substack{j\in [t]\\ j\neq i}}\left(\sum_{\substack{F\in \mathcal{F}_i\\  
        				j\notin F}}\sum_{x\in F}|\mathcal{F}_j(x)|+\sum_{\substack{F\in \mathcal{F}_i\\  
        				j\in F}}\left(\sum_{\substack{x\in F\\x\neq j}}|\mathcal{F}_j(x)|+|\mathcal{F}_j(j)|\right)\right)\notag\\
        		 &+\sum_{i\in [t]}\sum_{F\in \mathcal{F}_i}\left(\sum_{\substack{x\in F\\x\neq i}}|\mathcal{F}_i(x)|+|\mathcal{F}_i(i)|\right)\label{eq-10}
        	\end{align}
            where the third equality holds by Proposition~\ref{prop1}(iv). 
            Now we perform scaling on each term in \eqref{eq-10}. 
            \begin{itemize}
            	\item 
            	For 
            	$\sum_{\substack{F\in \mathcal{F}_i\\ j\notin F}}\sum_{x\in F}|\mathcal{F}_j(x)|$, since $x\in F$ and $j\notin F$, we have $x\neq j$. As each set in $\mathcal{F}_j(x)$ contains at least two fixed elements $j$ and $x$, $|\mathcal{F}_j(x)|\leq \binom{n-2}{k-2}$ and then $$\sum_{\substack{F\in \mathcal{F}_i\\ j\notin F}}\sum_{x\in F}|\mathcal{F}_j(x)|\leq \sum_{\substack{F\in \mathcal{F}_i\\ j\notin F}} k \binom{n-2}{k-2}.$$
            	\item 
            	For $\sum_{\substack{F\in \mathcal{F}_i\\  
                j\in F}}\sum_{\substack{x\in F\\x\neq j}}|\mathcal{F}_j(x)|$, since each set in $\mathcal{F}_j(x)$ contains at least two fixed elements $j$ and $x$, $|\mathcal{F}_j(x)|\leq \binom{n-2}{k-2}$. Also, as $x,j\in F$ and $x\neq j$, $x$ can be selected in $(k-1)$ ways. So we have $$\sum_{\substack{F\in \mathcal{F}_i\\  
                j\in F}}\sum_{\substack{x\in F\\x\neq j}}|\mathcal{F}_j(x)|\leq \sum_{\substack{F\in \mathcal{F}_i\\  
                j\in F}}(k-1)\binom{n-2}{k-2}.$$
                \item For $\sum_{\substack{F\in \mathcal{F}_i\\  
                		j\in F}}|\mathcal{F}_j(j)|$, by the definition of $\mathcal{F}_j$, we have 
                $$\sum_{\substack{F\in \mathcal{F}_i\\  
                		j\in F}}|\mathcal{F}_j(j)|=\sum_{\substack{F\in \mathcal{F}_i\\  
                		j\in F}}|\mathcal{F}_j|=|\mathcal{F}_i(j)||\mathcal{F}_j|.$$
                \item
                For $\sum_{F\in \mathcal{F}_i}\sum_{\substack{x\in F\\x\neq i}}|\mathcal{F}_i(x)|$, since $i\in F$, by the definition of $\mathcal{F}_i$,  $x$ can be selected in $(k-1)$ ways. Also, as each set in $\mathcal{F}_j(x)$ contains at least two fixed elements $j$ and $x$, $|\mathcal{F}_j(x)|\leq \binom{n-2}{k-2}$. So we have 
                $$\sum_{F\in \mathcal{F}_i}\sum_{\substack{x\in F\\x\neq i}}|\mathcal{F}_i(x)|\leq \sum_{F\in \mathcal{F}_i} (k-1)\binom{n-2}{k-2}=(k-1)\binom{n-2}{k-2}|\mathcal{F}_i|.$$
                \item For $\sum_{F\in \mathcal{F}_i}|\mathcal{F}_i(i)|$, by the definition of $\mathcal{F}_i$, we have 
                $$\sum_{F\in \mathcal{F}_i}|\mathcal{F}_i(i)|=\sum_{F\in \mathcal{F}_i}|\mathcal{F}_i|=|\mathcal{F}_i|^2.$$
            \end{itemize} 
            Thus, by \eqref{eq-10} and above scalings, we can obtain that
            \begin{align}
            	\omega(\mathcal{F},\mathcal{F})\leq&\sum_{i\in [t]}\sum_{\substack{j\in [t]\\ j\neq i}}\left(\sum_{\substack{F\in \mathcal{F}_i\\  
            			j\notin F}}k\binom{n-2}{k-2}+\sum_{\substack{F\in \mathcal{F}_i\\  
            			j\in F}}(k-1)\binom{n-2}{k-2}+|\mathcal{F}_i(j)||\mathcal{F}_j|\right)\notag\\
            	&+\sum_{i\in [t]}\left((k-1)\binom{n-2}{k-2}|\mathcal{F}_i|+|\mathcal{F}_i|^2\right),\notag\\
            	=&\sum_{i\in [t]}\sum_{\substack{j\in [t]\\ j\neq i}}\left(\sum_{\substack{F\in \mathcal{F}_i\\  
            			j\notin F}}k\binom{n-2}{k-2}+\sum_{\substack{F\in \mathcal{F}_i\\  
            			j\in F}}(k-1)\binom{n-2}{k-2}\right)+\sum_{i\in [t]}\sum_{\substack{j\in [t]\\ j\neq i}}|\mathcal{F}_i(j)||\mathcal{F}_j|\notag\\
            	&+\sum_{i\in [t]}\left((k-1)\binom{n-2}{k-2}|\mathcal{F}_i|+|\mathcal{F}_i|^2\right).\notag\\
            	\leq& \sum_{i\in [t]}(t-1)k\binom{n-2}{k-2}|\mathcal{F}_i|+\sum_{i\in [t]}\sum_{\substack{j\in [t]\\ j\neq i}}|\mathcal{F}_i(j)||\mathcal{F}_j|\notag\\
            	&+\sum_{i\in [t]}\left((k-1)\binom{n-2}{k-2}|\mathcal{F}_i|+|\mathcal{F}_i|^2\right),\label{eq-11}
            \end{align} 
            where the last inequality holds because
            $$
            \sum_{\substack{F\in \mathcal{F}_i\\ j\notin F}}k\binom{n-2}{k-2}+\sum_{\substack{F\in \mathcal{F}_i\\ j\in F}}(k-1)\binom{n-2}{k-2}\leq \sum_{F\in\mathcal{F}_i}k\binom{n-2}{k-2}=k\binom{n-2}{k-2}|\mathcal{F}_i|.
            $$
            Also, by \eqref{eq-12}, $\sum_{i\in [t]}|\mathcal{F}_i|=m$. Then by \eqref{eq-11} we have
            \begin{align}
            	\omega(\mathcal{F},\mathcal{F})\leq & (t-1)k\binom{n-2}{k-2}m+\sum_{j\in [t]}\sum_{\substack{i\in [t]\\ i\neq j}}|\mathcal{F}_j||\mathcal{F}_i(j)|+(k-1)\binom{n-2}{k-2}m+\sum_{i\in [t]}|\mathcal{F}_i|^2\notag\\
            	\leq & (tk-1)\binom{n-2}{k-2}m+\sum_{j\in [t]}(t-1)|\mathcal{F}_j|\binom{n-2}{k-2}+|\mathcal{F}_1|\sum_{i\in [t]}|\mathcal{F}_i|\notag\\
            	= &(tk-1)\binom{n-2}{k-2}m+(t-1)\binom{n-2}{k-2}m+|\mathcal{F}_1|m\notag\\
            	< &t(k+1)\binom{n-2}{k-2}m+|\mathcal{F}_1|m.\label{eq-13}
            \end{align}
            Furthermore, by Claim~\ref{claimm} and \eqref{eq5}, we can obtain the following upper bound of $\omega(\mathcal{F},\mathcal{F})$ from \eqref{eq-13}:
            $$
            \omega(\mathcal{F},\mathcal{F})<\frac{tk(k+1)}{\delta n}m^2+\left(\frac{m}{t}+\frac{tk^2}{\delta n}\right)m=\left(\frac{1}{t}+\frac{tk(2k+1)}{\delta n}\right)m^2.
            $$
            By Claim~\ref{claim6}, $t<12r$ and then 
            $$
            \omega(\mathcal{F},\mathcal{F})<\left(\frac{1}{t}+\frac{12k(2k+1)r}{\delta n}\right)m^2.
            $$
            Now, using Proposition~\ref{prop1}(v) and Claim~\ref{claim1}, we provide a lower bound for $\omega(\mathcal{F},\mathcal{F})$, that is  $\omega(\mathcal{F},\mathcal{F})=2\omega(\mathcal{F})+km\geq \frac{1}{r}m^2$. 
            
            Therefore, $\frac{1}{r}<\frac{1}{t}+\frac{12k(2k+1)r}{\delta n}$. Note that $\delta n> 12k(2k+1)(1+r)r^2$. This implies that $$\frac{12k(2k+1)r}{\delta n}<\frac{1}{r(r+1)}=\frac{1}{r}-\frac{1}{r+1}$$
            and then $\frac{1}{r}<\frac{1}{t}+\frac{1}{r}-\frac{1}{r+1}$. So $t<r+1$, as desired.
        \end{proof}
        Since $\binom{n}{k}-\binom{n-r+1}{k}<m=|\mathcal{F}|$, any cover of $\mathcal{F}$ has size at least $r$. By Claim~\ref{claim7}, $t=r$ and then $\mathcal{F}$ has a cover of $r$ elements. Now we have finished the proof of Lemma~\ref{lem1}.
    \end{proof}
    
    The following lemma shows if $\mathcal{F}_0$ is an extremal family of size $m$, then the complement of $\mathcal{F}_0$ is also an extremal family of size $\binom{n}{k}-m$. Hence, Theorem~\ref{thm-main} and Lemma~\ref{lem2} derive Theorem~\ref{thm-cor}.
    
    \begin{lem}\label{lem2}
    	Let $\mathcal{F}_0\subset \binom{[n]}{k}$ of size $m$. Assume that $\mathcal{F}_0$ is the unique extremal family (up to isomorphic) maximizing $\omega(\mathcal{F})$ among all families $\mathcal{F}\subset \binom{[n]}{k}$ of size $m$. Then $\binom{[n]}{k}\backslash\mathcal{F}_0$ is the unique extremal family (up to isomorphic) maximizing $\omega(\mathcal{G})$ among all families $\mathcal{G}\subset \binom{[n]}{k}$ of size $\binom{n}{k}-m$.
    \end{lem}
    \begin{proof}
    	For any $\mathcal{G}\subset \binom{[n]}{k}$ of size $\binom{n}{k}-m$, let $\mathcal{F}:=\binom{[n]}{k}\backslash\mathcal{G}$. Note that $|\mathcal{F}|=m$. We have $\omega(\mathcal{F})\leq \omega(\mathcal{F}_0)$ and then $\sum_{x\in [n]}|\mathcal{F}(x)|^2\leq \sum_{x\in [n]}|\mathcal{F}_0(x)|^2$, with equality if and only if $\mathcal{F}=\mathcal{F}_0$ up to isomorphic. Since each element $x\in [n]$ is in exact $\binom{n-1}{k-1}$ sets in $\binom{[n]}{k}$, by Proposition~\ref{prop1}(ii),
    	\begin{align*}
    		2\omega(\mathcal{G})=&\sum_{x\in [n]}|\mathcal{G}(x)|^2-k\left(\binom{n}{k}-m\right)\\
    		=&\sum_{x\in [n]}\left(\binom{n-1}{k-1}-|\mathcal{F}(x)|\right)^2-k\left(\binom{n}{k}-m\right)\\
    		=&n\binom{n-1}{k-1}^2-2km\binom{n-1}{k-1}+\sum_{x\in [n]}|\mathcal{F}(x)|^2-k\left(\binom{n}{k}-m\right)\\
    		\leq & n\binom{n-1}{k-1}^2-2km\binom{n-1}{k-1}+\sum_{x\in [n]}|\mathcal{F}_0(x)|^2-k\left(\binom{n}{k}-m\right)\\
    		=& \sum_{x\in [n]}\left(\binom{n-1}{k-1}-|\mathcal{F}_0(x)|\right)^2-k\left(\binom{n}{k}-m\right)\\
    		=& 2\omega\left (\binom{[n]}{k}\backslash\mathcal{F}_0\right).
    	\end{align*}
        Thus, $\omega(\mathcal{G})\leq \omega(\binom{[n]}{k}\backslash\mathcal{F}_0)$, with equality if and only if $\mathcal{G}=\binom{[n]}{k}\backslash\mathcal{F}_0$ up to isomorphic.
    \end{proof}

\section{Proof of Theorem~\ref{thm-main}}
    In our proof, we will use induction on $n, k$ and $m$. Before considering the base case of the induction, we first show how the induction works.
      
    Our goal is to prove Theorem~\ref{thm-main} holds for fixed $(n, k, m)$, that is
     
    {\bf Theorem~\ref{thm-main}A.}
    {\it  Let $n,k$ be fixed positive integers. Assume that $m$ is a fixed positive integer with $(n,k)$-cascade coefficients $(r_1,\ldots,r_s)$ and  
    \begin{align}
    	n>36k(2k+1)(k+r_s)r_s^2. \label{eq-n}
    \end{align}	
    Let $\mathcal{F}\subset \binom{[n]}{k}$ be an extremal family of size $m$. If for each $2\leq t\leq s$,
    \begin{align}
    	\frac{12k(2k+1)(k+1)}{n-r_s}<\delta_{n,k,m,t}=\frac{\sum_{i=t}^s\binom{n-r_i}{k-i+1}}{\binom{n-t+1}{k-t+1}}<1-\frac{12k(2k+1)(k+1)}{n-r_s}, \label{eq-d}
    \end{align}
    then $\mathcal{F}=\mathcal{L}_{n,k}^{m}$ up to isomorphic.
    }
    
    Assume that $\mathcal{F}$ is an extremal family of size $m$. Let $(r_1,\ldots,r_s)$ be the uniquely determined $(n,k)$-cascade coefficients given by \eqref{eqc}:
    $$
    \binom{n}{k}-m=\binom{n-r_1}{k}+\binom{n-r_2}{k-1}+\cdots+\binom{n-r_{s}}{k-s+1},1\leq r_1<\cdots<r_{s}\leq n-1.
    $$
    By the uniqueness of $k$-cascade of $\binom{n}{k}-m$, it is clear that that
    \begin{align}
    	\binom{n}{k}-\binom{n-r_1+1}{k}<m\leq \binom{n}{k}-\binom{n-r_1}{k}\label{eq-m}
    \end{align}
    and $n>36k(2k+1)(k+r_s)r_s^2\geq 36k(2k+1)(k+r_1)r_1^2$. We divide our discussion into two cases.
    
    {\bf Case 1.} $\mathcal{F}$ contains a full star. 
    Without loss of generality, assume that $\mathcal{F}(1)$ is a full star. 
    Then $\mathcal{F}(1)=\mathcal{F}_1=\{F\in \binom{[n]}{k}: 1\in F\}\subseteq \mathcal{F}$, which implies $m=|\mathcal{F}|\geq \binom{n-1}{k-1}$. First we consider the case $r_1=1$. In this case, by \eqref{eq-m}, $m\leq \binom{n-1}{k-1}$ and then $m= \binom{n-1}{k-1}$. So $\mathcal{F}=\mathcal{F}(1)=\mathcal{L}_{n,k}^m$, as desired. Now assume that $r_1\geq 2$. Note that 
    \begin{align}
    	\omega(\mathcal{F})=&\omega(\mathcal{F}_1)+\omega(\mathcal{F}\backslash\mathcal{F}_1)+\omega(\mathcal{F}_1,\mathcal{F}\backslash\mathcal{F}_1)\notag\\
    	=&\omega(\mathcal{F}_1)+\omega(\mathcal{F}\backslash\mathcal{F}_1)+\sum_{F\in \mathcal{F}\backslash\mathcal{F}_1}\sum_{x\in F}|\mathcal{F}_1(x)|.\label{eq:1}
    \end{align}

    Since $\mathcal{F}_1$ is a family independent of the structure of $\mathcal{F}$, $\omega(\mathcal{F}_1)$ depends only on $n$ and $k$. Also, for any $x\in F$, because $F\notin \mathcal{F}_1$, we have $x\neq 1$. Then $|\mathcal{F}_1(x)|=\binom{n-2}{k-2}$. Thus, $\sum_{F\in \mathcal{F}\backslash\mathcal{F}_1}\sum_{x\in F}|\mathcal{F}_1(x)|=|\mathcal{F}\backslash\mathcal{F}_1|k\binom{n-2}{k-2}=\left(m-\binom{n-1}{k-1}\right)k\binom{n-2}{k-2}$. Thus, if we want to maximize $\omega(\mathcal{F})$, by \eqref{eq:1}, it is sufficient to maximize $\omega(\mathcal{F}\backslash\mathcal{F}_1)$. 
    
    It should be noted that $\mathcal{F}\backslash\mathcal{F}_1$ is a family in $\binom{[n]\backslash\{1\}}{k}$ of size $m_0:=m-\binom{n-1}{k-1}$. Suppose that we have already proven Theorem~\ref{thm-main}A for $(n-1,k,m_0)$. Then $(\mathcal{F}\backslash\mathcal{F}_1)=\mathcal{L}_{n-1,k}^{m_0}$ maximizes $\omega(\mathcal{F}\backslash\mathcal{F}_1)$. Since $\mathcal{L}_{n-1,k}^{m_0}$ is isomorphic to the family consisting of the first $m_0$ sets in the lexicographical ordering in $\binom{[n]\backslash\{1\}}{k}$, we have $\mathcal{F}=\mathcal{L}_{n,k}^m$ up to isomorphic, as desired. 
    
    Hence, it is sufficient to prove Theorem~\ref{thm-main}A for $(n-1,k,m_0)$. We also need to verify that $(n-1,k,m_0)$ satisfies the conditions \eqref{eq-n} and \eqref{eq-d} of Theorem~\ref{thm-main}A. By \eqref{eqc}, it can be calculated that 
    \begin{align*}
    	\binom{n-1}{k}-m_0=&\binom{n-1}{k}-m+\binom{n-1}{k-1}\\
    	=&\binom{n}{k}-m\\
    	=&\binom{n-r_1}{k}+\binom{n-r_2}{k-1}+\cdots+\binom{n-r_{s}}{k-s+1}\\
    	=&\binom{n-1-(r_1-1)}{k}+\binom{n-1-(r_2-1)}{k-1}+\cdots+\binom{n-1-(r_{s}-1)}{k-s+1}.
    \end{align*}
    Then $(r_1-1,r_2-1,\ldots,r_s-1)$ are the $(n-1,k)$-cascade coefficients of $m_0$.
    This implies that $\delta_{n-1,k,m_0,t}=\delta_{n,k,m,t}$.
    
    By \eqref{eq-n}, we have $n-1>12k(2k+1)(k+r_s-1)(r_s-1)^2$.
    Also, using \eqref{eq-d}, for each $2\leq t\leq s$, we have
    $$\frac{12k(2k+1)(k+1)}{n-1-(r_s-1)}<\delta_{n-1,k,m_0,t}<1-\frac{12k(2k+1)(k+1)}{n-1-(r_s-1)}.$$
    So $(n-1,k,m_0)$ satisfies the conditions \eqref{eq-n} and \eqref{eq-d} of Theorem~\ref{thm-main}A.
    
    {\bf Case 2.} $\mathcal{F}$ contains no full star.
    
    In this case, since $n>36k(2k+1)(k+r_s)r_s^2$, by Lemma~\ref{lem1}, $\mathcal{F}$ has a cover of $r_1$ elements provided $r_1\geq 2$. Now we consider the case $r_1=1$.
    
    By \eqref{eqc} and \eqref{eqd} , it can be calculated that 
    \begin{align*}
    	\delta_{n,k,m,2}=\frac{\sum_{i=2}^s\binom{n-r_i}{k-i+1}}{\binom{n-1}{k-1}}
    	=\frac{\binom{n}{k}-\binom{n-r_1}{k}-m}{\binom{n-1}{k-1}}=\frac{\binom{n}{k}-\binom{n-1}{k}-m}{\binom{n-1}{k-1}}=1-\frac{m}{\binom{n-1}{k-1}}.
    \end{align*}
    Recall that $\delta_{n,k,m,2}$ satisfies \eqref{eq-d}. Then 
    $$\frac{m}{\binom{n-1}{k-1}}\cdot n=(1-\delta)n>\frac{12k(2k+1)(k+1)n}{n-r_s}>12k(2k+1)(k+1).$$
    Again, by Lemma~\ref{lem1}, $\mathcal{F}$ has a cover of $1$ elements. Thus,  $\mathcal{F}$ has a cover of $r_1$ elements regardless of whether $r_1=1$ or $r_1\geq 2$.
    
    Without loss of generality, assume that $[r_1]$ is a cover of $\mathcal{F}$. By \eqref{eq-m}, $m\leq \binom{n}{k}-\binom{n-r_1}{k}$. If $m=\binom{n}{k}-\binom{n-r_1}{k}$, then $r_1=r_s$ and $\mathcal{F}=\{F\in\binom{n}{k}:F\cap [r_1]\neq \emptyset\}$, which contains a full star. Thus in what follows, we assume that $m<\binom{n}{k}-\binom{n-r_1}{k}$. 
    
    Let $\mathcal{A}:=\{A\in \binom{[n]}{k}:A\cap [r_1]\neq \emptyset\}$ and $\mathcal{G}:=\mathcal{A}\backslash\mathcal{F}$. Then by Proposition~\ref{prop1}(v), we have $\omega(\mathcal{G},\mathcal{G})=2\omega(\mathcal{G})-k|\mathcal{G}|$. So
    \begin{align*}
    	\omega(\mathcal{A})=&\omega(\mathcal{F})+\omega(\mathcal{G})+\omega(\mathcal{G},\mathcal{F})\\
    	=&\omega(\mathcal{F})+\omega(\mathcal{G})+\omega(\mathcal{G},\mathcal{A})-\omega(\mathcal{G},\mathcal{G})\\
    	=&\omega(\mathcal{F})-\omega(\mathcal{G})+\omega(\mathcal{G},\mathcal{A})-k|\mathcal{G}|,
    \end{align*}
    that is 
    \begin{align}
    \omega(\mathcal{F})=\omega(\mathcal{A})+\omega(\mathcal{G})-\omega(\mathcal{G},\mathcal{A})+k|\mathcal{G}|.\label{eq:2}
    \end{align}
    Since $\mathcal{A}$ is a family independent of the structure of $\mathcal{F}$, $\omega(\mathcal{A})$ depends only on $n,k$ and $r_1$. Also, we have $|\mathcal{G}|=\binom{n}{k}-\binom{n-r_1}{k}-m$. Thus, by \eqref{eq:2}, if we want to maximize $\omega(\mathcal{F})$, it is sufficient to maximize $\omega(\mathcal{G})-\omega(\mathcal{G},\mathcal{A})$. In what follows, we will determine the minimum value of $\omega(\mathcal{G},\mathcal{A})$ and the maximum value of $\omega(\mathcal{G})$ separately.
    
    By a direct calculation,
    $$
    \omega(\mathcal{G},\mathcal{A})=\sum_{G\in\mathcal{G}}\omega(G,\mathcal{A})=\sum_{G\in\mathcal{G}}\left(\sum_{x\in G\cap [r_1]}|\mathcal{A}(x)|+\sum_{x\in G\backslash [r_1]}|\mathcal{A}(x)|\right).
    $$
    Note that $|\mathcal{A}(x)|=|\mathcal{A}(1)|=\binom{n-1}{k-1}$ if $x\in [r_1]$ and $|\mathcal{A}(x)|=|\mathcal{A}(r_1+1)|=\binom{n-1}{k-1}-\binom{n-r_1-1}{k-1}$ if $x\notin [r_1]$. Combining with the fact $|\mathcal{A}(1)|>|\mathcal{A}(r_1+1)|$ and $|G\cap [r_1]|+|G\backslash[r_1]|=|G|=k$, we have 
    \begin{align*}
    	\omega(\mathcal{G},\mathcal{A})=& \sum_{G\in\mathcal{G}}\left(|G\cap[r_1]||\mathcal{A}(1)|+|G\backslash[r_1]||\mathcal{A}(r_1+1)|\right)\\
    	\geq& |\mathcal{G}|\left(|\mathcal{A}(1)|+(k-1)|\mathcal{A}(r_1+1)|\right)\\
    	=& \left(\binom{n}{k}-\binom{n-r_1}{k}-m\right)\left(|\mathcal{A}(1)|+(k-1)|\mathcal{A}(r_1+1)|\right)\\
    	=&\left(\binom{n}{k}-\binom{n-r_1}{k}-m\right)\left(k\binom{n-1}{k-1}-(k-1)\binom{n-r_1-1}{k-1}\right),
    \end{align*}
    where the inequality holds because $|G\cap[r_1]|\geq 1$. This implies that $\omega(\mathcal{G},\mathcal{A})$ reaches its minimum value only if $|G\cap[r_1]|= 1$ for any $G\in \mathcal{G}$. 
    
    Now, we consider the maximum value of $\omega(\mathcal{G})$. Actually, $\mathcal{G}$ is a family in $\binom{[n]}{k}$ of size $m':=\binom{n}{k}-\binom{n-r_1}{k}-m$. Let $m'':=\binom{n-1}{k-1}-m'$ and $\mathcal{L}^c:=\{F\in \binom{[n]}{k}:1\in F\text{ and }F\notin \mathcal{L}_{n,k}^{m''}\}$. Since $\binom{n}{k}-\binom{n-r_1+1}{k}<m<\binom{n}{k}-\binom{n-r_1}{k}$, we have $0<m'<\binom{n-r_1+1}{k}-\binom{n-r_1}{k}=\binom{n-r_1}{k-1}\leq \binom{n-1}{k-1}$ and then $0<m''<\binom{n-1}{k-1}$. 
    
    Suppose we have already proven Theorem~\ref{thm-main}A for $(n-1,k-1,m'')$.  In what follows, we claim that $\omega(\mathcal{G})$ reaches its maximum value only if $\mathcal{G}=\mathcal{L}^c$ up to isomorphic.
    
    Since $m''<\binom{n-1}{k-1}$, by the definition of the lexicographic ordering, $\{1\}$ is a cover of $\mathcal{L}_{n,k}^{m''}$. Thus $\mathcal{L}^c\cup \mathcal{L}_{n,k}^{m''}=\{F\in \binom{[n]}{k}:1\in F\}$ and $|\mathcal{L}^c|=m'$. In other words, $\mathcal{L}^c$ is the family consisting of the largest $m'$ sets in the lexicographic ordering among the sets in $\binom{[n]}{k}$ that containing 1.
    
    Let $(\mathcal{L}^c-1)$ and $(\mathcal{L}_{n,k}^{m''}-1)$ be the families of sets obtained by removing 1 from each set in $\mathcal{L}^c$ and $\mathcal{L}_{n,k}^{m''}$ respectively. Then $(\mathcal{L}^c-1)$ is the complement of $(\mathcal{L}_{n,k}^{m''}-1)$ in $\binom{[n]\backslash\{1\}}{k-1}$. By the induction hypothesis, we have already proven Theorem~\ref{thm-main}A for $(n-1,k-1,m'')$. Thus among all families $\mathcal{F}'\subset \binom{[n]\backslash\{1\}}{k-1}$ of size $m''$, $\mathcal{L}_{n,k}^{m''}-1$ is the unique extremal family maximizing $\omega(\mathcal{F}')$. Then by Lemma~\ref{lem2}, among all families $\mathcal{G}'\subset \binom{[n]\backslash\{1\}}{k-1}$ of size $m'$, $\mathcal{L}^c-1$ is the unique extremal family maximizing $\omega(\mathcal{G}')$.

    Also, let $\mathcal{G}_0$ is an extremal family in $\binom{[n]}{k}$ of size $m'$. Then $\omega(\mathcal{G}_0)\geq\omega(\mathcal{G})$. Recall that $\delta_{n,k,m,2}$ satisfies \eqref{eq-d}. So using \eqref{eqc} and \eqref{eqd}, we have
    $$ \frac{12k(2k+1)(k+1)}{n-r_s}<\delta_{n,k,m,2}=\frac{\sum_{i=2}^s\binom{n-r_i}{k-i+1}}{\binom{n-1}{k-1}}
    =\frac{\binom{n}{k}-\binom{n-r_1}{k}-m}{\binom{n-1}{k-1}}=\frac{m'}{\binom{n-1}{k-1}}.$$
    This implies that $\frac{m'}{\binom{n-1}{k-1}}\cdot n>12k(2k+1)(k+1)$. Since $m'<\binom{n-1}{k-1}$, by Lemma~\ref{lem1}, $\mathcal{G}_0$ has a cover of 1 element. Without loss of generality, assume that $\{1\}$ is a cover of $\mathcal{G}_0$.
    Let $(\mathcal{G}_0-1)$ be the family by removing 1 from each set in $\mathcal{G}_0$. Then $(\mathcal{G}_0-1)$ is a family in $\binom{[n]\backslash\{1\}}{k-1}$ of size $m'$. Using Proposition~\ref{prop1}(ii), we have 
    $$
    \omega(\mathcal{G}_0)=\sum_{x\in [n]\backslash\{1\}}\binom{|\mathcal{G}_0(x)|}{2}+\binom{|\mathcal{G}_0(1)|}{2}=\omega(\mathcal{G}_0-1)+\binom{m'}{2}\leq \omega(\mathcal{L}'-1)+\binom{m'}{2}=\omega(\mathcal{L}^c).
    $$
    Thus, we have $\omega(\mathcal{G})\leq \omega(\mathcal{G}_0)\leq \omega(\mathcal{L}^c)$, with equality if and only if $\mathcal{G}=\mathcal{L}^c$ up to isomorphic, as desired.
    
    Moreover, since $m''=\binom{n-1}{k-1}-m'>\binom{n-1}{k-1}-\binom{n-r_1}{k-1}$, the minimum element in the $m''$-th set in $(\mathcal{L}^{m''}_{n,k}-1)$ must be larger than $r_1$. Then each set in $(\mathcal{L}^c-1)$ has no element in $[r_1]$. Let $(\mathcal{L}^c-1+r_1)$ be the family obtained by replacing $1$ with $r_1$ from each set in $\mathcal{L}^c$. Then $(\mathcal{L}^c-1+r_1)$ is isomorphic to $\mathcal{L}^c$, which means $\omega(\mathcal{G})$ reaches its maximum value provided $\mathcal{G}=(\mathcal{L}^c-1+r_1)$. Beside, $\omega(\mathcal{G},\mathcal{A})$ reaches its minimum value provided $\mathcal{G}=(\mathcal{L}^c-1+r_1)$ because $|G\cap [r_1]|=|\{r_1\}|=1$ for each $G\in \mathcal{G}$. Hence, $\omega(\mathcal{G})-\omega(\mathcal{G},\mathcal{A})$ reaches its maximum value only if  $\mathcal{G}=(\mathcal{L}^c-1+r_1)$ up to isomorphic.
    
    Note that $\mathcal{G}=(\mathcal{L}^c-1+r_1)$ is the family consisting of the largest $m'$ set in the lexicographic ordering among the sets in $\binom{[n]}{k}$, which have the minimum element $r_1$. So $\omega(\mathcal{F})$ reaches its maximum value only if $\mathcal{F}=\mathcal{A}\backslash(\mathcal{L}^c-1+r_1)=\mathcal{L}_{n,k}^{m}$, as desired.
    
    Hence, it is sufficient to prove Theorem~\ref{thm-main}A for $(n-1,k-1,m'')$. We also need to verify that $(n-1,k-1,m'')$ satisfies the conditions \eqref{eq-n} and \eqref{eq-d} of Theorem~\ref{thm-main}A. By \eqref{eqc}, it can be calculated that 
    \begin{align*}
    	\binom{n-1}{k-1}-m''=&\binom{n-1}{k-1}-\binom{n-1}{k-1}+m'\\
    	=&\binom{n}{k}-\binom{n-r_1}{k}-m\\
    	=&\binom{n-r_2}{k-1}+\binom{n-r_3}{k-2}+\cdots+\binom{n-r_{s}}{k-s+1}\\
    	=&\binom{n-1-(r_2-1)}{k-1}+\binom{n-1-(r_3-1)}{k-1}+\cdots+\binom{n-1-(r_{s}-1)}{k-s+1}.
    \end{align*}
    Since $r_2\geq r_1+1\geq 2$, $(r_2-1,r_3-1,\ldots,r_s-1)$ are the $(n-1,k-1)$-cascade coefficients of $m''$. By \eqref{eq-n}, we have $n-1>36(k-1)(2k-1)(k+r_s-2)(r_s-1)^2$. By \eqref{eqd},
    $$\delta_{n-1,k-1,m',t}=\frac{\sum_{i=t}^{s-1}\binom{n-r_{i+1}}{k-i}}{\binom{n-t}{k-t}}=\frac{\sum_{i=t+1}^s\binom{n-r_{i}}{k-i+1}}{\binom{n-t}{k-t}}=\delta_{n,k,m,t+1}.$$
    Thus, using \eqref{eq-d} and the fact $12k(2k+1)k>12(k-1)(2k-1)(k-1)$, for each $2\leq t\leq s-1$, we have
    $$\frac{12(k-1)(2k-1)k}{n-1-(r_s-1)}<\delta_{n-1,k-1,m',t}<1-\frac{12(k-1)(2k-1)k}{n-1-(r_s-1)}.$$
    
    {\bf Base case of the induction.} Based on the argument above, in each step of induction, we turn $(n,k,m)$ into either $(n-1,k,m_0)$ or $(n-1,k-1,m'')$. Because $n$ is much larger than $k$ and $m$ is not increasing at each step, in our base case, we only need to prove one of the following cases:
    \begin{wst}
    	\item Theorem~\ref{thm-main}A for $(n,1,m)$;
    	\item Theorem~\ref{thm-main}A for $(n,k,1)$;
    	\item Theorem~\ref{thm-main}A for $(n,k,m)$, $\mathcal{F}$ contains a full star and $m\leq \binom{n-1}{k-1}$.
    \end{wst} 
    Clearly, all of them are trivial cases. Hence, we have finished the proof of Theorem~\ref{thm-main}. \qed
\section{Remark}
    In this paper, we determine that  $\omega(\mathcal{F})\leq \omega(\mathcal{L}_{n,k}^{m})$ for some special $m$ provided $n$ is sufficiently large. This result give a natural extension of the enumeration problem of the disjoint pairs in a family of sets, and also give an extremal hypergraph graph maximizing the sum of squares of degrees.
    
    We also want to mention a combinatorial explanation of \eqref{eqc}. Let $\{r_1,\ldots,r_k\}$ be the $m$-th set in the lexicographic ordering on $\binom{[n]}{k}$ with $1\leq r_1<r_2<\cdots<r_k\leq n$. Then we claim that $m$ and $\{r_1,\ldots,r_s\}$ satisfies \eqref{eqc}, where $r_s$ is the last element with $r_s\leq n-k+s-1$. This implies that $(r_{s+1},r_{s+2},\ldots,r_k)=(n-k+s+1,n-k+s+2,\ldots,n)$.
    
    To prove this, we divide $\mathcal{L}_{n,k}^m$ into two families $\mathcal{A}_{1}\cup \mathcal{L}_1$, where $\mathcal{A}_1$ consists of sets with some $r<r_1$ as the minimum element, and $\mathcal{L}_1$ consists of sets with $r_1$ as the minimum element. It is clear that $|\mathcal{A}_1|=\binom{n}{k}-\binom{n-r_1+1}{k}$. 
    
    Further, we remove $r_1$ from each set in $\mathcal{L}_1$ and divide the obtained family into two subfamilies $\mathcal{A}_{2}\cup \mathcal{L}_2$, where $\mathcal{A}_2$ consists of sets with some $r_1<r<r_2$ as the minimum element, and $\mathcal{L}_2$ consists of sets with $r_2$ as the minimum element. Again, it is clear that $|\mathcal{A}_2|=\binom{n-r_1}{k-1}-\binom{n-r_2+1}{k-1}$. 
    
    Then we remove $r_2$ from each set in $\mathcal{L}_2$ and divide the obtained family into two subfamilies. Continue this partitioning process until the reminding subfamily $\mathcal{L}_s$ contains all sets in $\left\{L\in\binom{[n]\backslash[r_s-1]}{k-s+1}:r_s \text{ is the minimum element in } L\right\}$ and then $|\mathcal{L}_s|=\binom{n-r_s}{k-s}$. Then we have
    \begin{align*}
    	m=&|\mathcal{A}_1|+|\mathcal{A}_2|+\cdots+|\mathcal{A}_{s-1}|+|\mathcal{A}_s|+|\mathcal{L}_s|\\
    	=&\binom{n}{k}-\binom{n-r_1+1}{k}+\binom{n-r_1}{k-1}-\binom{n-r_2+1}{k-1}+\cdots-\binom{n-r_s+1}{k-s+1}+\binom{n-r_s}{k-s}\\
    	=&\binom{n}{k}-\binom{n-r_1}{k}-\binom{n-r_2}{k-1}-\cdots-\binom{n-r_s}{k-s+1}.
    \end{align*}
    Thus, $m$ and $\{r_1,\ldots,r_s\}$ satisfies \eqref{eqc}. In other words, if $(r_1,\ldots,r_s)$ is determined by \eqref{eqc}, then the last set in $\mathcal{L}_{n,k}^m$ is actually $\{r_1,\ldots,r_s,n-k+s+1,n-k+s+2,\ldots,n\}$.
    
    By the definition~\eqref{eqd} of $\delta_{n,k,m,t}$, we have
    \begin{align*}
    	\delta_{n,k,m,t}\cdot\binom{n-t+1}{k-t+1}=&\sum_{i=t}^s\binom{n-r_i}{k-i+1}\\
    	=&\sum_{i=t}^s\left(\binom{n-r_i+1}{k-i+1}-\binom{n-r_i}{k-i}\right)\\
    	=&\binom{n-r_t+1}{k-t+1}-\sum_{i=t}^{s-1}\left(\binom{n-r_i}{k-i}-\binom{n-r_{i+1}+1}{k-i}\right)-\binom{n-r_s}{k-s}\\
    	=&\binom{n-r_t+1}{k-t+1}
    	-\sum_{i=t}^{s}|\mathcal{A}_i|-|\mathcal{L}_s|\\
    	=&\binom{n-r_t+1}{k-t+1}+\sum_{i=1}^{t-1}|\mathcal{A}_i|-m.
    \end{align*}
    Note that $|\mathcal{L}_1|=m-|\mathcal{A}_1|$, $|\mathcal{L}_i|=|\mathcal{L}_{i-1}|-|\mathcal{A}_i|$ for $2\leq i\leq t-1$. So we have $|\mathcal{L}_{t-1}|=m-\sum_{i=1}^{t-1}|\mathcal{A}_i|$. Thus, 
    $$\delta_{n,k,m,t}=\frac{\binom{n-r_t+1}{k-t+1}-|\mathcal{L}_{t-1}|}{\binom{n-t+1}{k-t+1}}.$$
      
    Now recall the condition of Theorem~\ref{thm-main}:
    $$\frac{12k(2k+1)(k+1)}{n-r}<\delta_{n,k,m,t}<1-\frac{12k(2k+1)(k+1)}{n-r}.$$ The range of $\delta_{n,k,m,t}$ implies that $|\mathcal{L}_{t-1}|$ can not be too small or too large. In fact, this condition may be extended but $\delta_{n,k,m,t}$ can not too close to $0$ or $1$.
    
    For example, let $k$ be a fixed integer. Assume that $n$ is sufficiently large and $m$ is the positive integer with $(n,k)$-cascade coefficients $(1,2,\ldots,k-2,k-1,2k)$, that is
    \begin{align*}
    	m=&\binom{n}{k}-\binom{n-1}{k}-\binom{n-2}{k-1}-\cdots-\binom{n-(k-1)}{k-(k-1)+1}-\binom{n-2k}{k-k+1}\\
    	=&\binom{n-k+1}{1}-\binom{n-2k}{1}\\
    	=&k+1.
    \end{align*}
    By \eqref{eqd}, we have
    $$\delta_{n,k,m,2}=\frac{\sum_{i=2}^{k}\binom{n-i}{k-i+1}}{\binom{n-2+1}{k-2+1}}=\frac{\binom{n}{k}-\binom{n-1}{k}-m}{\binom{n-1}{k-1}}=1-\frac{k+1}{\binom{n-1}{k-1}}>1-\frac{12k(2k+1)k}{n-2k}.$$
    It is clear that $\mathcal{F}=\binom{[k+1]}{k}$ is an extremal family of size $m$ maximizing $\omega(\mathcal{F})$. However, $\mathcal{F}$ is not isomorphic to $\mathcal{L}^m_{n,k}$.
    
\section{Declaration of competing interest}

The authors declare that they have no known competing financial interests or personal relationships that could have appeared to influence the work reported in this paper.

\section{Acknowledgements}
Erfei Yue is supported by ERC Advanced grant GeoScape, No. 882971.


\begin{thebibliography}{99}
	\bibitem{Ahlswede2} R. Ahlswede, G. O. H. Katona: Graphs with maximal number of adjacent pairs of edges, {\it Acta Mathematica Academiae Scientiarum Hungarica}, 32(1-2) (1978) 97-120.
	\bibitem{Ahlswede} R. Ahlswede: Simple hypergraphs with maximal number of adjacent
	pairs of edges, {\it Journal of Combinatorial Theory, Series B}, 28(2) (1980) 164-167.
	\bibitem{Bey} C. Bey: An upper bound on the sum of squares of degrees in a hypergraph, {\it Discrete Mathematics}, 269 (2003) 259-263.
	\bibitem{Bollobas} B. Bollob\'{a}s, I. Leader: Set systems with few disjoint pairs, {\it Combinatorica}, 23(2003) 559-570.
	\bibitem{Caen} D. de Caen: An upper bound on the sum of squares of degrees in a graph, {\it Discrete Mathematics}, 185 (1998) 245-248.
	\bibitem{chDas} K. C. Das: Maximizing the sum of the squares of the degrees of a graph, {\it Discrete Mathematics}, 285 (2004) 57-66.
	\bibitem{Das} S. Das, W. Gan, B. Sudakov: The minimum number of disjoint pairs in set systems and related problems, {\it Combinatorica}, 36(6) (2016) 623-660.
	\bibitem{Erdos} P. Erd\H{o}s, C. Ko, R. Rado: Intersection theorems for systems of finite sets, {\it Quarterly Journal of Mathematics}, 12 (1) (1961) 313-320.
	\bibitem{Frankl} P. Frankl: On the minimum number of disjoint pairs in a family of finite sets, {\it Journal of Combinatorial Theory, Series A}, 22 (1977) 249-251.
	\bibitem{Kong} X. Kong, G. Ge: On an inverse problem of the Erd\H{o}s-Ko-Rado type theorems, arXiv:2004.01529.
\end{thebibliography}
\end{document}